\documentclass[11pt,a4paper,twoside]{article}
\usepackage{latexsym, amsfonts}
\usepackage[latin1]{inputenc}

\usepackage{amssymb}
\usepackage{dsfont}
\usepackage{graphicx}
\usepackage{bbm}
\usepackage[usenames,dvipsnames]{color}
\usepackage{amsthm}
\usepackage{amsmath}
\usepackage{amsfonts}
\usepackage{enumerate}
\usepackage{ifpdf}
\usepackage[pdftex]{thumbpdf}       
\usepackage[textsize=small]{todonotes}
\graphicspath{{pictures/}}

\usepackage{eepic}
\usepackage{multirow}

\newcommand{\e}{{\mathrm e}}

\theoremstyle{plain}
\newtheorem{theorem}{Theorem}[section]

\newtheorem{lemma}[theorem]{Lemma}

\newtheorem{op}[theorem]{Open problem}

\theoremstyle{remark}

\newtheoremstyle{maria}{}{}{}{0cm}{\itshape}{}{\newline}{}
\theoremstyle{maria}

\numberwithin{equation}{section}
\numberwithin{theorem}{section}

\usepackage{fancyhdr}

\newcommand{\Cbold} {{\mathbb C}}

\newcommand{\Nbold} {{\mathbb N}}

\newcommand{\Qbold} {{\mathbb Q}}
\newcommand{\Rbold} {{\mathbb R}}

\newcommand{\Tbold} {{\mathbb T}}

\newcommand{\Zbold} {{\mathbb Z}}

\newcommand{\Wcal}   {\mathcal{W}}

\def\1{{\mathchoice {1\mskip-4mu\mathrm l}      
{1\mskip-4mu\mathrm l}
{1\mskip-4.5mu\mathrm l} {1\mskip-5mu\mathrm l}}}
\newcommand{\indic}[1]{\1_{\{#1\}}}
\newcommand{\indicwo}[1]{\1_{#1}}

\newcommand{\expec}{{\mathbb{E}}}

\newcommand{\R}{\Rbold}
\newcommand{\Z}{\Zbold}
\newcommand{\Zd}{\Zbold^d}

\newcommand{\sss}   { \scriptscriptstyle }

\newcommand {\vep}{\varepsilon}

\newcommand{\eqarray}   {\begin{eqnarray}}
\newcommand{\enarray}   {\end{eqnarray}}
\newcommand{\lbeq}[1]  {\label{e:#1}}
\newcommand{\refeq}[1] {\eqref{e:#1}}
\newcommand{\eq}{\begin{equation}}
\newcommand{\en}{\end{equation}}
\newcommand{\ben}{\begin{enumerate}}
\newcommand{\een}{\end{enumerate}}
\newcommand{\eqn}[1]{\begin{equation} #1 \end{equation}}
\newcommand{\eqan}[1]{\begin{align} #1 \end{align}}

\newcommand{\nn}{\nonumber}

\renewcommand{\to}{\rightarrow}

\def\mA[#1]{ {\bf \hat A}(#1)}
\def\mD[#1]{{\bf \hat D}(#1)}
\def\v1{{\vec 1}}
\def\mJ{{\bf J}}
\def\mI{{\bf I}}
\def\mC{{\bf C}}
\def\ve[#1]{ {e}_{#1}}

\def\vb{{\vec b}}
\newcommand{\degree}{m}
\newcommand{\ii}{{\mathrm i}}
\newcommand{\Tm}{T_{{\rm mix}}}
\newcommand{\lambdahat}{\hat{\lambda}}

\title{Non-backtracking random walk\\
{\normalsize (Preprint)}}

\author{Robert Fitzner\thanks{Department of Mathematics and
        Computer Science, Eindhoven University of Technology,
        5600 MB Eindhoven, The Netherlands.
        {\tt R.J.Fitzner@tue.nl},{\tt rhofstad@win.tue.nl}}
        \and
        Remco van der Hofstad$^*$
    }
\date{December 11, 2012}
\begin{document}
\maketitle
\begin{abstract}
We consider non-backtracking random walk (NBW) in the nearest-neighbor setting
on the $\Zd$-lattice and on tori. We evaluate the eigensystem of the
$\degree\times \degree$-dimensional transition matrix of NBW where $\degree$
denote the degree of the graph. We use its eigensystem to show a functional central
limit theorem for NBW on $\Zd$ and to obtain estimates on the
convergence towards the stationary distribution for NBW on the torus.
\end{abstract}

\section{Introduction}
The non-backtracking walk (NBW) is a simple random walk that is
conditioned not to jump back along the edge it has just traversed.
NBW can be viewed as a Markov chain on the set of directed edges,
where the edge has the interpretation of being the last edge traversed
by the NBW, but we will not rely on this interpretation.
We study NBWs on $\Zd$ in the nearest-neighbor setting, and
NBWs on tori in various settings, derive the transition matrix
for the NBW and analyse its eigensystem in Fourier space.
We use this to study asymptotic properties of the NBW, such as
its Green's function and a functional central limit theorem (CLT)
on $\Zd$, and its convergence towards the stationary distribution
on the torus. In particular, our analysis allows us to give
an explicit formula for the Fourier transform of the number
of $n$-step NBWs traversing fixed positions at given times. We use this to
prove that the finite-dimensional distributions of the NBW
displacements, after diffusive rescaling,  converge to
those of Brownian motion. By an appropriate tightness argument,
this proves a functional CLT. We further evaluate the
Fourier transform of the NBW $n$-step transition
probabilities on the torus to identify when the NBW transition
probabilities are close to the stationary distribution.
Our paper is inspired by the study of various
high-dimensional statistical mechanical models. For example,
our derivation allows us to give detailed estimates of
the probability that NBW on a torus is at a given vertex,
a fact used in the analysis of hypercube percolation
in \cite{HofNac11}.

NBWs have been investigated on finite graphs in \cite{AloBenLubSod07},
in particular expanders, where it was shown that NBWs mix faster than ordinary random
walks. See also \cite{AloLub09}, where a Poisson limit was proved for
the number of visits of an $n$-step NBW to a vertex in an $r$-regular graph
of size $n$. In the nearest-neighbor setting on $\Z^d$, NBWs
were investigated in \cite[Section 5.3]{MadSla93}, where
an explicit expression of its Green's function and many
related properties are derived (see also \cite[Exercise 3.8]{Slad06}).
Noonan \cite{Noon98} investigates the generating function
of NBWs, and his results also apply to walks that avoid their last 7 previous
positions (i.e., with memory up to 8), and were used to improve the
known upper bounds on the SAW connective constant. See
\cite{PonTit00} for an extension up to memory 22 for $d=2$, further
improving the upper bound on the SAW connective constant.
The methods in \cite{Noon98, PonTit00} allow to compute the
generating function of the number of memory-$k$ SAWs for
appropriate values of $k$, but do not investigate the
number of memory-$k$ SAWs ending in a particular position in $\Z^d$.
Finally, \cite{OrtWoe07} studies relations between the
exponential growth of the transition probabilities of NBW
and non-amenability of the underlying graph.
For interesting connections to zeta functions on graphs, which allow one to
study the number of NBWs of arbitrary length ending in the
starting point, we
refer the reader to \cite{KotSun00}.

This paper is oganized as follows. In Section \ref{secmodel} we investigate
NBW on $\Z^d$ and in Section \ref{secTorus} we study NBW on tori.

\section{Non-backtracking random walk on $\Zd$}
\label{secmodel}
\subsection{The setting}
An \emph{$n$-step nearest-neighbor simple random walk} (SRW)
on $\Zd$ is an ordered $(n+1)$-tuple $\omega=(\omega_0,\omega_1,\omega_2,\dots, \omega_n)$,
with $\omega_i\in\Zbold$ and $\|\omega_i-\omega_{i+1}\|_1=1$, where
$\|x\|_1=\sum_{i=1}^d |x_i|$. We always take $\omega_0=(0,0,\dots,0)$.
The step distribution of SRW is given by
    \begin{eqnarray}
    \lbeq{DefDn}
    D(x)=\frac 1 {2d} \indic{\|x\|_1=1},
    \end{eqnarray}
where $\indicwo{A}$ is the indicator of the event $A$. If an $n$-step SRW $\omega$ additionally
satisfies that $\omega_i\not=\omega_{i+2}$, then we call the walk a
\emph{non-backtracking walk} (NBW). As the problem of NBW is trivial for $d=1$,
we always assume that $d\geq 2$. For the NBW we also count walks conditioned not
to take their first step in a certain direction $\iota$. We exclusively
use the Greek letters $\iota$ and $\kappa$ for values in
$\{-d,-d+1,\dots,-1,1,2,\dots,d\}$ and denote by $\ve[\iota]\in\Zbold^d$
the unit vector in direction $\iota$, i.e.\
$(\ve[\iota])_\kappa=\text{sign}(\iota) \delta_{|\iota|,\kappa}$
(beware of the minus sign when $\iota$ is negative, which is somewhat
different from the usual choice of a unit vector).
Let $b_n(x)$ be the number of $n$-step NBWs with $\omega_n=x$,
and $b^{\iota}_{n}(x)$ the number of $n$-step
NBWs $\omega$ with $\omega_n=x$ and $\omega_1 \not =\ve[\iota]$.
For $n\geq 1$, the following relations between these objects hold:
    \begin{eqnarray}
    \lbeq{NBWRecScheme1}
    b_{n}(x)&=&\sum_{\iota\in\{\pm1,\dots,\pm d\}} b^{-\iota}_{n-1}(x-\ve[\iota]),\\
    \lbeq{NBWRecScheme2}
    b_{n}(x)&=&b^{\iota}_{n}(x) + b^{-\iota}_{n-1}(x-\ve[\iota])
		\quad \forall \iota,\\
    \lbeq{NBWRecScheme3}
    b^{\iota}_{n}(x)&=& \sum_{\kappa\in\{\pm1,\dots,\pm d\}\setminus\{\iota\} } b^{-\kappa}_{n-1}(x-\ve[\kappa]).
    \end{eqnarray}
We analyse $b_{n}$ and $b^{\iota}_{n}$ using \emph{Fourier theory}.
For an absolutely summable function $f\colon \Zbold^d\mapsto\Cbold$,
we define its Fourier transform by	
    \begin{eqnarray}
    \label{defFourier}
    \hat f (k) =\sum_{x\in\Zbold} f(x) \e^{\ii k\cdot x}\qquad (k\in [-\pi,\pi]^d),
    \end{eqnarray}
where $k\cdot x=\sum^d_{s=1} k_sx_s$ and $\ii$ denotes the imaginary unit.
We use $k$ exclusively to
denote values in the Fourier dual space $[-\pi,\pi]^d$.
For $f,g\colon\Zbold^d\mapsto\Cbold$ we denote their convolution
by $f \star g$, i.e.,
    \begin{eqnarray}
    (f \star g)(x) =\sum_{y\in\Zbold} f(y)g(x-y),
    \end{eqnarray}
and note that the Fourier transform of $f\star g$ is given by $\hat f\hat g$.
Applying the Fourier transformation to \refeq{DefDn}-\refeq{NBWRecScheme3} yields
    \begin{eqnarray}
    \lbeq{DefDk}
    \hat D(k)&=&\frac 1 {d} \sum_{i=1}^d \cos(k_\iota),\\
    \lbeq{NBWRecSchemeFourier1}
    \hat b_{n}(k)&=&\sum_{\iota\in\{\pm1,\dots,\pm d\}} \hat b^{-\iota}_{n-1}(k)\e^{\ii k_\iota},\\
    \lbeq{NBWRecSchemeFourier2}
    \hat b_{n}(k)&=&\hat b^{\iota}_{n}(k) + \hat b^{-\iota}_{n-1}(k) \e^{\ii k_\iota}
		\quad \forall \iota,\\
    \lbeq{NBWRecSchemeFourier3}
    \hat b^{\iota}_{n}(k)&=& \sum_{\kappa\in\{\pm1,\dots,\pm d\}\setminus\{\iota\} } \hat b^{-\kappa}_{n-1}(x)\e^{\ii k_\kappa}.
    \end{eqnarray}
In our further analysis, we use $\Cbold^{2d}$-valued and $\Cbold^{2d}\times\Cbold^{2d}$-valued
functions. For a clear distinction between scalar-, vector- and matrix-valued
quantities, we always write $\Cbold^{2d}$-valued functions with a vector
arrow (e.g.\ $\vec v$) and matrix-valued functions with bold capital letters
(e.g. ${\bf M}$). We do not use $\{1,2,\dots,2d\}$ for the indices of
the elements of a vector or a matrix, but use $\{-d,-d+1,\dots,-1,1,2,\dots,d\}$
instead. Here, for a negative index $\iota\in\{-d,-d+1,\dots,-1\}$ and a vector
$k\in[-\pi,\pi]^d$, we define $k_\iota:=-k_{|\iota|}$.

We denote the identity matrix by $\mI\in\Cbold^{2d\times 2d}$ and the all-one
vector by $ \v1 =(1,1,\dots,1)^T\in\Cbold^{2d}$.
Moreover we define the matrices $\mC,\mJ\in\Cbold^{2d\times 2d}$ by $(\mC)_{\iota,\kappa}=1$  and  $(\mJ)_{\iota,\kappa}=\delta_{\iota,-\kappa}$. To characterize the displacement of a step in a
direction $\iota$, we define the diagonal matrix $\mD[k]$ with the entries $(\mD[k])_{\iota,\iota}=\e^{\ii k_\iota}$. We define the vector $\vb_{n}(k)$ with entries $\left( \vb_{n}(k) \right)_\iota=\hat b^\iota_n(k)$.
Then, we can rewrite \refeq{NBWRecSchemeFourier1}-\refeq{NBWRecSchemeFourier3} as
    \begin{eqnarray}
    \lbeq{NBWRecSchemeFourierVec1}
    \hat b_{n}(k)&=&\v1^T\mD[-k]\vb_{n-1}(k),\\
    \lbeq{NBWRecSchemeFourierVec2}
    \hat b_{n}(k)\v1&=&\vb_{n}(k) + \mD[k]\mJ \vb_{n-1}(k),\\
    \lbeq{NBWRecSchemeFourierVec3}
    \vb_{n}(k)&=& (\mC-\mJ)\mD[-k]\vb_{n-1}(k)= \left((\mC-\mJ)\mD[-k]\right)^{n}\v1.
    \end{eqnarray}
We define the transition matrix
	\eqn{
	\lbeq{defA}
	\mA[k]=(\mC-\mJ)\mD[-k],
	}
so that
$(\mA[k])_{\iota,\kappa}=\e^{-ik_{\iota}}(1-\delta_{\iota,-\kappa})$.
With this notation in hand, we are ready to identify the NBW Green's function.

\subsection{The Green's function}
\label{subsecTwoFunction}
We start by deriving a formula for the NBW Green's function using the relations in
\refeq{NBWRecSchemeFourierVec1}-\refeq{NBWRecSchemeFourierVec1}. While these results are
not novel, the analysis presented here is efficient and simple.
We define the NBW Green's function as the generating function
of $\hat b_n$ and $\hat b^\iota_n$:
    \begin{eqnarray}
    \label{defGjz}
     \hat B_z(x)&:=&\sum_{n=0}^\infty \hat b_n(k) z^n, \qquad \qquad
     \hat B^\iota_z(x):=\sum_{n=0}^\infty \hat b^\iota_n(k) z^n\\
     {\vec B}_z(k)^T&:=&(\hat B^1_z(k),\hat B^{-1}_z(k),\hat B^2_z(k),\dots,\hat B^{-d}_z(k)),
    \end{eqnarray}
where $\vec{y}^T$ denotes the transpose of the vector $\vec{y}\in \R^d$.
By \refeq{NBWRecSchemeFourierVec1}-\refeq{NBWRecSchemeFourierVec2},
    \begin{eqnarray}
    \lbeq{NBWGenVec1}
    \hat B_{z}(k)&=&1+z\v1^T\mD[-k]\vec B_{z}(k),\\
    \lbeq{NBWGenVec2}
    \hat B_{z}(k)\v1&=&\vec B_{z}(k) + z\mD[k]\mJ \vec B_{z}(k)\  \Rightarrow\  \vec B_{z}(k) =\left[\mI+z\mD[k]\mJ\right]^{-1}\v1 \hat B_{z}(k),
    \end{eqnarray}
Using $\mD[k]\mJ\mD[k]\mJ= \mI$, it is easy to check that
    \begin{eqnarray}
    \lbeq{invIDJ}
    \left[\mI+z \mD[k]\mJ\right]^{-1}&=& \frac 1 {1-z^2} \left(\mI-z \mD[k]\mJ\right),
    \end{eqnarray}
and we use \refeq{NBWGenVec1}-\refeq{invIDJ} to obtain
    \begin{eqnarray}
    \lbeq{NBWGenSolved}
    \hat B_{z}(k)&=&\frac 1 {1- z\v1^T\mD[-k]\left[\mI+z\mD[k]\mJ\right]^{-1}\v1}
    = \frac {1-z^2} {1+(2d-1)z^2-2dz\hat D(k)}.
    \end{eqnarray}
Note in particular that
	\eqn{
	\hat B_{z}(k)=\frac{1-z^2}{1+(2d-1)z^2} \hat{C}_{\mu_z}(k),
	}
where $\mu_z=2dz/(1+(2d-1)z^2)$ and $\hat{C}_{\mu}(k)=1/[1-\mu\hat{D}(k)]$ is the SRW Green's function. By \refeq{NBWGenVec2},
    \begin{eqnarray*}
    \vec B_{z}(k)&=&\frac 1 {1-z^2} \left[\mI-z\mD[k]\mJ\right] \v1 \hat B_{z}(k)= \frac {\left[\mI-z\mD[k]\right] \v1} {1+(2d-1)z^2-2dz\hat D(k)},
    \end{eqnarray*}
so that
    \begin{eqnarray}
    \lbeq{NBWGenIotaSolved}
    \hat B^\iota_{z}(k)&=& \frac {1-z\e^{\ii k_\iota}} {1+(2d-1)z^2-2dz\hat D(k)}.
    \end{eqnarray}

\subsection{The transition matrix}
\label{secMatrix}
\paragraph{The eigensystem.}
We start by evaluating the transition matrix \refeq{defA} by characterizing its eigenvalues
and eigenvectors:

\begin{lemma}[Dominant eigenvalues]
\label{EVrighthard}
For $d\geq 2$ and $k\in (-\pi,\pi)^d$, let $\mA[k]$ be the matrix given in \refeq{defA}. Then
    \eqn{
    \lambdahat_{\pm 1}=\lambdahat_{\pm 1}(k)= d \hat D(k)\pm\sqrt{(d\hat D(k))^2 -(2d-1)}
    }
are eigenvalues of $\mA[k]$.
For $k\not=(0,0,\dots,0)^T$, the right eigenvectors ${\vec v}^{\sss(\pm1)}$ to the eigenvalue $\lambdahat_{\pm1}$ are given by
    \begin{eqnarray}
    \lbeq{vpm-def}
     {\vec v}^{\sss(\pm 1)} = \lambdahat_{\pm 1} {\vec 1} \pm \mD[k]{\vec 1}.
    \end{eqnarray}
For $k=(0,0,\dots,0)^T$, the eigenvectors are given by ${\vec v}^{\sss(1)} (0)=(2d-2){\v1}$ and
${\vec v}^{\sss(-1)} :=(1,-1,0,0,0,\dots,0)\in\mathbb{Z}^{2d}$.
\end{lemma}
As we will see below, Lemma \ref{EVrighthard} yields the two most
important eigenvalues. When $\lambdahat_+=\lambdahat_-$,
which occurs when $(d\hat D(k))^2 -(2d-1)=0$,
it turns out that ${\vec v}^{\sss(1)}$ has geometric multiplicity
$1$, and that ${\vec 1}$ is a generalized eigenvector satisfying
$\mA[k]{\vec 1}={\vec v}^{\sss(1)}+\lambdahat_+{\vec 1}$.
We continue by computing the remaining eigenvalues and -vectors:

\begin{lemma} [Simple eigenvalues]
\label{EVrightsimple} For $d\geq 2$, $k\in (-\pi,\pi)^d$, let $\mA[k]$ be the
matrix given in \refeq{defA} and ${\vec e}_\iota\in\Cbold^{2d}$ the
$\iota$th unit vector, i.e., $({\vec e}_\iota )_\kappa=\delta_{\iota,\kappa}$ for $\kappa\in\{\pm1,\dots,\pm d\}$.\\
For $\iota\in \{2,3,\dots,d\}$, let
\begin{eqnarray*}
{\vec v}^{\sss(\iota)}&=& (1+\e^{\ii k_\iota})(\e^{\ii k_\iota}{\vec e}_1+{\vec e}_{-1})-(1+\e^{\ii k_1})(\e^{\ii k_\iota}{\vec e}_{\iota}+{\vec e}_{-\iota})\qquad \ \forall k\in[-\pi,\pi]^d\\
{\vec v}^{\sss(-\iota)}&=& (1-\e^{\ii k_\iota})(\e^{\ii k_1}{\vec e}_1-{\vec e}_{-1})-(1-\e^{\ii k_1})(\e^{\ii k_\iota}{\vec e}_{\iota}-{\vec e}_{-\iota})\qquad \forall k\in[-\pi,\pi]^d \text{ if } \e^{\ii k_1}\not=1,
\end{eqnarray*}
Then, ${\vec v}^{\sss(\pm \iota)}$ is an eigenvector of $\mA[k]$ to the eigenvalue $\mp 1$.
Both eigenvalues have a geometrical multiplicity of $d-1$ for all $k$.
\end{lemma}

Lemmas \ref{EVrighthard}--\ref{EVrightsimple} identify a collection of $2d$ independent eigenvectors, and thus the complete eigensystem, of $\mA[k]$. Now we prove these two lemmas:

\begin{proof}[Proof of Lemma \ref{EVrighthard}]
Let $\lambdahat\in\{\lambdahat_1,\lambdahat_{-1}\}$ and ${\vec v}=\lambdahat {\bf 1} -\mD[k]{\bf 1}$.
The values $\lambdahat_1$ and $\lambdahat_{-1}$ are the solutions of the quadratic equation
    \begin{eqnarray}
    \lbeq{quadend}
     \lambdahat^2= 2d\lambdahat \hat D(k)  -  (2d-1).
     \end{eqnarray}
Using $\mC\mD[-k]\v1=2d\hat D(k) \v1$ and $\mJ\mD[-k]=\mD[k]\mJ$, we compute
    \begin{eqnarray*}
    \mA[k] {\vec v}&=&(\mC-\mJ)\mD[-k](\lambdahat \mI-\mD[k])\v1=\left(2d \lambdahat\hat D(k)\mI-\lambdahat\mD[k]-(2d-1)\mI\right)\v1\\
    &\stackrel{\refeq{quadend}}=&\left(\lambdahat^2\mI -\lambdahat\mD[k]\right)\v1=\lambdahat \vec v.
    \end{eqnarray*}
This proves that $\vec v$ is a eigenvector of $\mA[k]$ corresponding to the eigenvalue $\lambdahat$ for all $k\neq 0$ and also for the case of $\iota=1$ for $k=0$.
For $k=(0,\dots,0)$ we note that $\lambdahat_{-1}(0)=1$ and see that
    \begin{eqnarray*}
    \mA[0]{\vec v}^{\sss(-1)}(0)=(\mC-\mJ){\vec v}^{\sss(-1)}(0)=-\mJ{\vec v}^{\sss(-1)}(0)={\vec v}^{\sss(-1)}(0).
    \end{eqnarray*}
\end{proof}

\begin{proof}[Proof of Lemma \ref{EVrightsimple}]
For $\iota \in\{1,2,\dots,d\}$, the vectors
    \begin{eqnarray*}
      \vec {u}^{\sss(\iota)}= \e^{\ii  k_\iota}{\vec e}_{\iota} +  {\vec e}_{-\iota}
    \qquad \text{and}\qquad
      \vec{u}^{\sss(-\iota)}&=&  \e^{\ii  k_\iota}{\vec e}_{\iota} - {\vec e}_{-\iota}
    \end{eqnarray*}
are eigenvectors of $\mJ\mD[-k]$, where $\vec {u}^{\sss(\pm \iota)}$  is associated to the eigenvalue $\pm 1$.
For $\iota \in\{2,3,\dots,d\}$, we define
    \eqn{
    {\vec v}^{\sss(\iota)}= \vec{u}^{\sss(1)} \sum_{\kappa} \vec{u}^{\sss(\iota)}_\kappa-\vec{u}^{\sss(\iota)} \sum_{\kappa} \vec{u}^{\sss(1)}_\kappa,
	\qquad\qquad
    \lbeq{diffchoiceEV}
    {\vec v}^{\sss(-\iota)}= \vec{u}^{\sss(-1)} \sum_{\kappa} \vec{u}^{\sss(-\iota)}_\kappa-\vec{u}^{\sss(-\iota)} \sum_{\kappa} \vec{u}^{\sss(-1)}_\kappa.
    }
By construction, ${\vec v}^{\sss(\iota)}$ and ${\vec v}^{\sss(-\iota)}$ are also eigenvalues of $\mJ\mD[-k]$.
For $\mC\mD[-k]$, we compute that
    \[
    \mC\mD[-k]\vec{u}^{\sss(\iota)}=\mC({\vec e}_{\iota} + \e^{\ii  k_\iota} {\vec e}_{-\iota})= \sum_\kappa \vec{u}^{\sss(\iota)}_\kappa \v1,\qquad\qquad
    \mC\mD[-k]\vec{u}^{\sss(-\iota)}=\mC({\vec e}_{\iota} - \e^{\ii  k_\iota} {\vec e}_{-\iota})= -\sum_\kappa \vec{u}^{\sss(-\iota)}_\kappa \v1,
    \]
so that
    \begin{eqnarray*}
     \mC\mD[-k]\vec{v}^{\sss(\iota)}=\mC\mD[-k]\vec{v}^{\sss(-\iota)}=0.
    \end{eqnarray*}
Knowing this, it follows that
    \begin{eqnarray}
    \lbeq{EVproberty1}
    \mA[k]\vec{v}^{\sss(\iota)}=(\mC-\mJ)\mD[-k]\vec{v}^{\sss(\iota)}=-\vec{v}^{\sss(\iota)},\\
    \lbeq{EVproberty2}
    \mA[k]\vec{v}^{\sss(-\iota)}=(\mC-\mJ)\mD[-k]\vec{v}^{\sss(\iota)}=\vec{v}^{\sss(-\iota)}.
    \end{eqnarray}
By \refeq{EVproberty1}, $\vec{v}^{\sss(\iota)}$ is an eigenvector for all $k$.
Since the set of vectors $(\vec{v}^{\sss(\iota)})_{\iota=2,3,\dots,d}$ is linearly
independent we know that the eigenvalue $-1$ has geometric multiplicity $d-1$.

From \refeq{EVproberty2}, we conclude the existence of $d-1$ linear independent
eigenvalue for $1$ only when $\e^{\ii k_\iota}\neq 1$ for all $\iota\in\{1,2,\dots,d\}$.
To prove that the eigenvalue $1$ has geometric multiplicity $d-1$ for all $k$,
we show how to choose $d-1$ linear independent eigenvectors when $\e^{\ii k_\kappa}= 1$
for a $\kappa\in\{1,2,\dots,d\}$. For this, let $S_1$ be the set of all
$\kappa\in\{1,2,\dots,d\}$ with $k_\kappa= 0$ and $S_2$ the set of all
$\kappa\in\{1,2,\dots,d\}$ with $k_\kappa\neq 0$.
Let $s_1$ and $s_2$ be the number of elements in $S_1$ and $S_2$. Then $s_1+s_2=d$.
For all $\iota\in S_1$, we define ${\vec v}^{\sss(-\iota)}=\ve[\iota]  -\ve[-\iota]$.

If $s_1=d$, then $k=0$. In Lemma \ref{EVrighthard}, we define for this case $\lambdahat_2=1$ and ${\vec v}^{\sss(-1)}=\ve[1]-\ve[-1]$.
Then $\{ \vec{v}^{\sss(-1)},\vec{v}^{\sss(-2)},\dots,\vec{v}^{\sss(-d)}\}$ is a set of
independent eigenvectors of $\mA[k]$ to the eigenvalue $1$.

If $s_1<d$, then let $\rho$ be the smallest number in $S_2$ and define
    \begin{eqnarray*}
    {\vec v}^{\sss(-\iota)}&=& \vec{u}^{\sss(-\rho)} \sum_{\kappa} \vec{u}^{\sss(-\iota)}_\kappa-\vec{u}^{\sss(-\iota)} \sum_{\kappa} \vec{u}^{\sss(-\rho)}_\kappa.
    \end{eqnarray*}
for all $\iota\in S_2\setminus\{\kappa\}$. Then it is easy to verify that the vectors $(\vec{v}^{\sss(-\iota)})_{\iota=2,3,\dots,d}$ are linearly independent and are
eigenvectors of $\mA[k]$ with eigenvalue $1$.
\end{proof}

We use the eigensystem of the matrix $\mA[k]$ to identify $\hat b_n(k)$
and $\vec b_n(k)$:

\begin{lemma}[NBW characterization]
\label{qncharac}
Let $d\geq 2$, $n\geq 1$ and $k\in(-\pi,\pi^d)$ such that $\lambdahat_1(k)\neq \lambdahat_{-1}(k)$. Then,
    \begin{eqnarray}
    \lbeq{qncharac1}
    \hat b_n(k)&=& 2d  \frac {\hat D(k)(\lambdahat^{n}_{1}(k)-\lambdahat^{n}_{-1}(k)) + \lambdahat^{n-1}_{-1}(k)-\lambdahat^{n-1}_{1}(k)}{\lambdahat_{1}(k)-\lambdahat_{-1}(k)},\\
    \lbeq{qncharac2}
    \vec b_n(k)&=&  \frac{\lambdahat^{n}_{1}(k) -\lambdahat^{n}_{-1}(k) }{\lambdahat_{1}(k) -\lambdahat_{-1}(k) }\mD[k]\v1 -\frac{\lambdahat^{n-1}_{1}(k) -\lambdahat^{n-1}_{-1}(k) }{\lambdahat_{1}(k) -\lambdahat_{-1}(k) }\v1.
    \end{eqnarray}
When $\lambdahat_{1}(k)=\lambdahat_{-1}(k)$,
    \eqn{
    \lbeq{qncharac1,2-spec}
    \hat b_n(k)=2d[(n-1)\lambdahat_{1}(k)^{n-2}+\hat{D}(k)n\lambdahat_{1}(k)^{n-1}],
    \qquad
    \vec b_n(k)=[(n+1)\lambdahat_{1}(k)^{n}+n\lambdahat_{1}(k)^{n-1}\mD[k]]\v1.
    }
\end{lemma}
Clealry, one can reprove \refeq{NBWGenSolved} and \refeq{NBWGenIotaSolved} using Lemma \ref{qncharac}.

\begin{proof}
For $\lambdahat_1(k)\neq \lambdahat_{-1}(k)$, we define
    \begin{eqnarray*}
    \alpha(k)&=&\frac 1 {\lambdahat_1(k)-\lambdahat_{-1}(k)}=\frac 1 {2 \sqrt{(d\hat D(k) )^2 -(2d-1)} }.
    \end{eqnarray*}
We can write
    \begin{eqnarray*}
    \v1 &=& \alpha(k) \left( \lambdahat_1(k) \mI -\lambdahat_{-1}(k) \mI + \mD[k]- \mD[k]\right)\v1= \alpha(k) \vec v^{\sss(1)}(k)-\alpha(k) \vec v^{\sss(-1)}(k).
    \end{eqnarray*}
Using \refeq{NBWRecSchemeFourierVec3} and the fact that $\vec v^{\sss(\pm1)}(k)$ are eigenvectors of $\mA[k]$
with eigenvalue $\lambdahat_{\pm}(k)$, we obtain
    \begin{eqnarray}
    \lbeq{qncharac3}
    \vec b_{n}(k)&=&\mA[k]^{n}\v1= \alpha(k)\left(\lambdahat^{n}_{1}(k)\vec v^{\sss(1)}(k) -\lambdahat^{n}_{-1}(k)\vec v^{\sss(-1)}(k) \right),
    \end{eqnarray}
which proves \refeq{qncharac2}. Combining \refeq{NBWRecSchemeFourierVec1} and \refeq{qncharac3} gives
    \begin{eqnarray}
    \lbeq{qncharac4}
    \hat b_{n}(k)&=&\alpha(k) \v1^T\mD[-k]\left(\lambdahat^{n}_{1}(k)\vec v^{\sss(1)}(k) -\lambdahat^{n}_{-1}(k)\vec v^{\sss(-1)}(k) \right).
    \end{eqnarray}
Inserting the definition of $\vec v^{\sss(\pm 1)}(k)$  gives \refeq{qncharac1}.
The proof of \refeq{qncharac1,2-spec} is similar, now using that
$\mA[k]{\vec 1}={\vec v}^{\sss(1)}+\lambdahat_+{\vec 1}$, which implies that
$\mA[k]^n{\vec 1}=n\lambdahat_+(k)^{n-1}{\vec v}^{\sss(1)}+\lambdahat_+(k)^{n}{\vec 1}$.
\end{proof}

\subsection{Central limit theorem}
\label{secCLT}
This section is devoted to the proof of a functional central
limit theorem for the NBW. The uniform measures on $n$-step NBWs
form a consistent family of measures,
so that there is a unique law that describes them as a process.
We let $\omega=(\omega_r)_{r\geq 0}$ be distributed according to this law.
For a NBW $\omega$ and $t\geq 0$, we define
    \begin{eqnarray}
    X_n(t)=\frac {\omega_{\lfloor nt\rfloor}} {\sqrt{n}},
    \end{eqnarray}
where $\lfloor x \rfloor$ denotes the integer part of $x\in \R$.

\begin{theorem}[Functional central limit theorem]
\label{FCLT}
The processes $(X_n(t))_{t\geq 0}$
converge weakly to $(B(t))_{t\geq 0}$, where $(B(t))_{t\geq 0}$ is
a Brownian motion with covariance matrix $\mI/(d-1)$.
\end{theorem}

Our proof of Theorem \ref{FCLT} is organized as follows.
We use Lemma \ref{qncharac} to prove a CLT for the endpoint in
Lemma \ref{lemmacnlimit}. We then
prove the convergence of the finite-dimensional distributions to the
Gaussian distribution (see Lemma \ref{FCLTpointwise}), followed by a
proof of tightness (see Lemma \ref{tightnessLemma}).
This implies Theorem \ref{FCLT}, see e.g.\ \cite[Theorem 15.6]{Bill95}.
We now fill in the details.

\begin{lemma}[Central limit theorem]
\label{lemmacnlimit}
Let $d\geq 2$ and $k\in [-\pi,\pi]^d$. Then,
    \begin{eqnarray}
    \lim_{n\to \infty} \expec[\e^{\ii k\cdot \omega_n/ \sqrt{n}}]&=& \e^{-\|k\|^2_2/(2d-2)},
    \end{eqnarray}
where $\|k\|_2^2=\sum_{i=1}^d k_i^2$ is the Euclidean norm of $k$.
\end{lemma}

Lemma \ref{lemmacnlimit} implies that the distribution of the endpoint of an
$n$-step NBW converges in distribution to a normal distribution
with mean zero and covariance matrix $(d-1)^{-1}{\mathbf I}$.
\begin{proof}
We can rewrite the expectation as
    \begin{eqnarray}
    \lbeq{Eqnscheme}
    \expec[\e^{\ii k\cdot \omega_n/ \sqrt{n}}] &=& \sum_{x\in\Zbold} \frac {b_n(x)} {\hat b_n(0)} \e^{\ii k\cdot x/ \sqrt{n}}=\frac {\hat b_{n}(k/\sqrt{n})}{2d(2d-1)^{n}}.
    \end{eqnarray}
As $n\rightarrow \infty$, we can assume without loss of generality that $k$ is small and therefore that $\lambdahat_1(k)>\lambdahat_{-1}(k)$. Then we can use \refeq{qncharac1} to compute the
limit of \refeq{Eqnscheme}. For $\iota={1,-1}$, we compute
    \begin{eqnarray}
    \lbeq{shownlimit}
    \lim_{n\to \infty}\frac {\lambdahat^{n-1}_\iota(k/\sqrt{n})}  {(2d-1)^{n-1}} \alpha(k/\sqrt{n})\frac {\v1^T\mD[-k/\sqrt{n}]{\vec v}^{\sss(\iota)}(k/\sqrt{n})} {2d}.
    \end{eqnarray}
We first consider the case $\iota=1$. The coefficient $\alpha(k)$
as well as $\vec v^{\sss(1)}(k)$ are continuous in a neighborhood of
$\vec{0}$, so we can directly compute
    \begin{eqnarray*}
    \lim_{n\to \infty} \alpha(k)\frac {\v1^T\mD[-k/\sqrt{n}]{\vec v}^{\sss(1)}(k/\sqrt{n})}  {2d}= \alpha(0) \frac {\v1^T\mD[0]{\vec v}^{\sss(1)}(0)}  {2d}=1.
    \end{eqnarray*}
Further, $\lambdahat_1(k)$ is differentiable in $k$, so we can Taylor expand $\lambdahat_1(k)$ at $0$
to obtain
    \begin{eqnarray*}
    \lambdahat_{1}(k)&=& 2d-1 -  \frac {2d-1}{2d-2} \|k\|^2_2   + O(\|k\|_2^4).
    \end{eqnarray*}
Therefore,
    \begin{eqnarray}
    \lbeq{lampart}
    \lim_{n\to \infty}\frac {\lambdahat^{n-1}_1 (k/\sqrt{n})}  {(2d-1)^{n-1}}&=&
    \lim_{n\to \infty} \left(1 -  \frac {1}{2d-2} \frac{ \|k\|^2_2}{n}   + O\left(\frac{\|k\|_4^4}{n^2}\right)\right)^{n-1}
    =\e^{-\|k\|^2_2/(2d-2)},
    \end{eqnarray}
so that the limit \refeq{shownlimit} for $\iota =1$ is given by $\e^{-\frac 1 {2d-2} \|k\|^2_2}$.\\
We next consider the case $\iota=-1$, for which we use that
$k\mapsto \lambdahat_{-1}(k)$ is continuous and $\lambdahat_{-1}(0)=1$,
Therefore, for $d\geq 2$,
    \begin{eqnarray}
    \lbeq{LambminusVanish}
    \lim_{n\to \infty}\frac {\lambdahat^{n-1}_{-1}(k/\sqrt{n})}  {(2d-1)^{n-1}}=0.
    \end{eqnarray}
The second factor in \refeq{shownlimit} can easily be bounded uniformly for small $k$,
so that for $\iota=-1$ the limit of \refeq{shownlimit} is zero.
\end{proof}

\begin{lemma}[Convergence of finite-dimensional distributions]
\label{FCLTpointwise}
For $d\geq 2$ and $N>0$, let $0=t_0< t_1<t_2< \dots < t_N= 1$ and $k^{\sss(r)}\in (-\pi,\pi]^d$ for $r=1,\dots,N$. Then,
    \begin{eqnarray}
    \lbeq{FCLTpointwiseStatement}
    \lim_{n \to \infty} \expec[\e^{\ii (\sum_{r=1}^N k^{\sss(r)}(\omega_{\lfloor t_r n \rfloor}-\omega_{\lfloor t_{r-1}n\rfloor })/\sqrt{n})}]&=&\e^{-\sum_{r=1}^N\|k^{\sss(r)}\|_2^2(t_r - t_{r-1} )/(2d-2)}.
    \end{eqnarray}
\end{lemma}

\begin{proof}{}
As we take the limit $n\to\infty$, without loss of generality we can assume that
$\eta_r(n):=\lfloor t_{r}n \rfloor - \lfloor t_{r-1}n \rfloor\geq 1$, $t_r>t_{r-1}$
and each $k^{\sss(r)}_n:=k^{\sss(r)}/\sqrt{n}$ is so small that
$\lambdahat_1(k)>\lambdahat_{-1}(k)$ for $r=1,\dots,N$ and $n\in \Nbold$.
Let $\Wcal_n$ be the set of all $n$-step NBW. For any function
$f\colon\mathbb{Z}^{d\times N} \mapsto \mathbb{C}$, we know that
    \begin{eqnarray}
    \nonumber
    \expec[ f(\omega_{\lfloor t_1 n \rfloor},\dots, \omega_{\lfloor t_N n \rfloor})]
    &=& \frac 1 {\hat b_n(0)} \sum_{\omega\in\Wcal_n} f(\omega_{\lfloor t_1 n\rfloor},\dots, \lfloor \omega_{t_N n\rfloor})\\
    \lbeq{cnf}
    &=& \frac 1 {\hat b_n(0)} \ \sum_{\omega\in \Wcal_n} \sum_{x_1,\dots,x_N\in\Zbold} f(x_1,\dots,x_N) \prod_{i=1,\dots, N}\delta_{x_i,\lfloor\omega_{t_i n}\rfloor}.
    \end{eqnarray}
Let $b_n^{\iota,\kappa}(x)$ be the number of $n$-step NBW $\omega$ with $\omega_1\neq\ve[\iota]$, $\omega_{n-1}=x+\ve[\kappa]$ and $\omega_n=x$.
We define the matrix $\hat{\bf B}_n(k)$ with entries $({\bf \hat B}_n(k))_{\iota,\kappa}=\hat{b}_{n}^{\iota, \kappa} (k)$. By a relation similar to \refeq{NBWRecSchemeFourier3}, we conclude that ${\bf \hat B}_{n}(k)={\bf \hat A}(k)^{n}$.
We fix $N$ points $x_1,\dots,x_N\in\Zbold$, then the number of NBW $\omega$ with $\omega_{\lfloor t_in\rfloor}=x_i$ for all $i=1,2,\dots,N$ is given by
    \begin{eqnarray}
    \lbeq{splittingIntoSmallPieces}
     \sum_{\stackrel{\iota_1,\dots, \iota_{N}}{\text{in }\{\pm1,\dots,\pm d\}}}b_{\eta_1(n)-1}^{\iota_0, \iota_1}(x_1+\ve[\iota_0])
     \prod_{r=2,\dots,N} b_{\eta_r(n)}^{\iota_{r-1}, \iota_r}(x_r-x_{r-1}).
    \end{eqnarray}
We insert $f(x_1,\dots,x_N)= \e^{\ii \sum_r k^{\sss(r)}_n (x_r-x_{r-1})}$ and obtain
    \begin{eqnarray}
    \nonumber
    \expec[\e^{\ii \sum_r k^{\sss(r)}_{n}(\omega_{n_r}-\omega_{n_{r-1}} )}]&=& \frac 1 {\hat b_n(0)}
    \sum_{\stackrel{\iota_1,\dots, \iota_{N}}{\text{in }\{\pm1,\dots,\pm d\}}} \hat b_{\eta_1(n)-1}^{\iota_1}(k^{\sss(1)}_n)\e^{-\ii k_{\iota_1}^{\sss(1)}/\sqrt{n}}
    \prod_{r=2}^{N} \hat b_{\eta_r(n)}^{\iota_{r-1}, \iota_r} (k^{\sss(r)}_{n})\\
    \nonumber
    &=&\frac 1 {\hat b_n(0)}  {\vec 1}^T \mD[-k^{\sss(1)}_n]{\bf \hat A}(k^{\sss(1)}_n)^{\eta_1(n)-1}
	\prod_{r=2}^N {\bf \hat B}_{\eta_r(n)}(k^{\sss(r)}_n) {\vec 1}\\
    \lbeq{FCLTchar2}
    &=&\frac 1 {\hat b_n(0)}  {\vec 1}^T\mD[-k^{\sss(1)}_n]{\bf \hat A}(k^{\sss(1)}_n)^{\eta_1(n)-1}
    \prod_{r=2}^N {\bf \hat A}(k^{\sss(r)}_n)^{\eta_{r}(n)}{\vec 1}.
    \end{eqnarray}
To proceed, we define, for $\tau\in\{1,2,\dots,N\}$,
    \begin{eqnarray*}
    \vec h_n(\tau)  &=& \frac 1 {\hat b_{\lfloor t_\tau n\rfloor}(0)}  {\vec 1}^T\mD[-k^{\sss(1)}_n]
    {\bf \hat A}(k^{\sss(1)}_n)^{\eta_1(n)-1} \cdots {\bf \hat A}(k^{\sss(\tau)}_n)^{\eta_{\tau}(n)},
    \qquad\qquad
    \vec h_n(0)  =\frac 1 {2d}  {\vec 1}^T \mD[-k^{\sss(1)}_n].
    \end{eqnarray*}
As $\sum^\tau_{r=1} \eta_r(n) \leq t_{\tau}n$, by construction of $\vec h_n(\tau)$,
for all ${\vec v}\in\mathbb{C}^{2d}$
    \begin{eqnarray}
    \nonumber
    |\vec h_n(\tau){\vec v}| &\leq& \expec\left[ \e^{\ii (\sum_{r=1}^\tau k^{\sss(r)}(\omega_{\lfloor t_r n \rfloor}-\omega_{\lfloor t_{r-1}n\rfloor })/\sqrt{n})}\sum_{\iota} \indic{\omega_{t_\tau n}-\omega_{t_\tau n-1}=\ve[\iota]} |\vec v_\iota| \right]\\
    \lbeq{matrixprodBounded}
    &\leq& \expec[\sum_{\iota} \indic{\omega_{t_\tau n}-\omega_{t_\tau n-1}=\ve[\iota]} |\vec v_\iota| ]\leq \|\vec v\|_{\infty},
    \end{eqnarray}
where we write  $\|v\|_{\infty}:=\max_{\iota}|v_\iota|$.
Therefore, we can rewrite
    \begin{eqnarray}
    \nonumber
    \refeq{FCLTchar2} &=&\vec h_n(N)\v1= \vec h_n(N-1) \left( \frac {\mA[k^{\sss(N)}_n]} {2d-1} \right)^{\eta_N(n)} \frac{\vec v^{\sss(1)}(k^{\sss(N)}_n)}{2d-2} +\vec h_n(N)\left(\v1-\frac{\vec v^{\sss(1)}(k^{\sss(N)}_n)}{2d-2}\right) \\
    \nonumber
    &=& \vec h_n(N-1) \left( \frac {\lambdahat_1(k^{\sss(N)}_n)} {2d-1} \right)^{\eta_N(n)} \frac {\vec v^{\sss(1)}(k^{\sss(N)}_n)}{2d-2} +\vec h_n(N)(\v1-\frac {\vec v^{\sss(1)}(k^{\sss(N)}_n)}{2d-2}) \\
    \lbeq{FCLTchar3a}
    &=& \vec h_n(0) \frac{\vec v^{\sss(1)}(k^{\sss(1)}_n)}{2d-2} \left( \frac {2d-1} {\lambdahat_1(k^{\sss(1)}_n)}  \right) \prod_{r=1}^N \left( \frac {\lambdahat_1(k^{\sss(r)}_n)} {2d-1} \right)^{\eta_{r}(n)}\\
    \lbeq{FCLTchar3b}
    &&+\sum_{r=1}^{N-1}\frac{\vec h_n(r)}{2d-2} \left( \vec v^{\sss(1)}(k^{\sss(r+1)}_n)-\vec v^{\sss(1)}(k^{\sss(r)}_n)\right) \prod_{i=r+1}^N \left( \frac {\lambdahat_1(k^{\sss(i)}_n)} {2d-1} \right)^{\eta_{i}(n)}\\
    \lbeq{FCLTchar3c}
    &&+\vec h_n(N)\left(\v1-\frac {\vec v^{\sss(1)}(k^{\sss(N)}_n)}{2d-2}\right).
    \end{eqnarray}

To compute the limit of \refeq{FCLTchar2}, we show that \refeq{FCLTchar3b}
and \refeq{FCLTchar3c} converge to zero and identify the limit of \refeq{FCLTchar3a}.
From \refeq{matrixprodBounded}, it follows that we can bound \refeq{FCLTchar3c} by
$\v1-\vec v^{\sss(1)}(k^{\sss(N)}_n)/(2d-2)$
and we know that $\vec v^{\sss(1)}(k^{\sss(N)}_n)\stackrel{n\to\infty}\rightarrow (2d-2)\v1$.
Thereby we conclude that \refeq{FCLTchar3c} converges to $0$.\\
To compute the limit of \refeq{FCLTchar3b} we note that $\lambdahat_1(k)\leq 2d-1$ and that $\vec v^{\sss(r+1)}(k^{\sss(r+1)}_n)-\vec v^{\sss(r)}(k^{\sss(r)}_n)\stackrel{n\to\infty}\rightarrow \vec 0$, so that
    \begin{eqnarray*}
    \lim_{n\to\infty}|\refeq{FCLTchar3b}|
    &\leq&\lim_{n\to\infty} \Big|\sum_{r=1}^{N-1}\frac{\vec h_n(r)}{2d-2} \left( \vec v^{\sss(1)}(k^{\sss(1)}_n)-\vec v^{\sss(1)}(k^{\sss(1)}_n)\right)\Big|\\
    &\stackrel{\refeq{matrixprodBounded}}\leq&
    \lim_{n\to\infty} \frac 1 {2d-2} \sum_{r=1}^{N-1}\| \vec v^{\sss(1)}(k^{\sss(1)}_n)-\vec v^{\sss(1)}(k^{\sss(1)}_n)\|_\infty=0.
    \end{eqnarray*}
To compute the limit of \refeq{FCLTchar3a}, we use that
    \begin{eqnarray*}
    \lim_{n\to\infty} \vec h_n(0) \frac{\vec v^{\sss(1)}(k^{\sss(1)}_n)}{2d-2}\frac {2d-1} {\lambdahat_1(k^{\sss(1)}_n)} &=& \lim_{n\to\infty}  \frac{\v1^T \mD[-k](\lambdahat_1(k^{\sss(1)}_n)\v1+\mD[k^{\sss(1)}_n]\v1)}{2d(2d-2)}\frac {2d-1} {\lambdahat_1(k^{\sss(1)}_n)}=1,
    \end{eqnarray*}
and, for each $r\in \{1, \ldots, N\}$,
    \begin{eqnarray*}
    \lim_{n\to\infty} \left( \frac {\lambdahat_1(k^{\sss(r)}_n)} {2d-1} \right)^{\eta_{r}(n)} &\stackrel{\refeq{lampart} }=&\e^{-\|k^{\sss(i)}\|^2_2(t_i-t_{i-1})/(2d-2)}.
    \end{eqnarray*}
Combining this yields that the limit of \refeq{FCLTchar3a} is the right-hand side of
\refeq{FCLTpointwiseStatement}, which completes the proof of Lemma \ref{FCLTpointwise}.
\end{proof}

To prove tightness we make use of the following lemma, which computes the
second moment of the end-point of NBW. The leading order in this result was alternatively proved in
\cite[(5.3.11)]{MadSla93} using residues.

\begin{lemma}[The second moment]
\label{limitedsecondmoment}
For $d\geq 2$ and $n\in\Nbold$,
    \begin{eqnarray*}
    \expec[\|\omega_n\|_2^2]=\frac{d}{d-1}n +\frac{4d-1}{2(d-1)^2}+
    \frac{d}{2(d-1)^2(2d-1)^{n-2}}.
    \end{eqnarray*}
\end{lemma}

\begin{proof}{}
We define the differential operator $\nabla^2:=\sum_{i=1}^d \left(\frac{\partial} {\partial k_i}\right)^2$ and see that
    \begin{eqnarray*}
    \label{secondMomentInitialObservation}
     \expec[\|\omega_n\|_2^2] = \frac 1 {\hat b_n(0)}\sum_{x\in \Zbold} b_n(x) \|x\|_2^2&=& -\frac 1 {\hat b_n(0)} \nabla^2 \hat b_n(0).
    \end{eqnarray*}
To compute the second derivative in a neighborhood of the origin, we recall \refeq{qncharac1}.
By Lemma \ref{qncharac}, for $k$ such that $\lambdahat_+(k)\neq \lambdahat_-(k)$,
    \begin{eqnarray}
    \hat b_n(k)&=&\sum_{\sigma\in\{-1,1\}} d\iota  \frac  {  \lambdahat^{n}_\sigma(k)\hat D(k) -\lambdahat^{n-1}_\sigma(k)}{\sqrt{(d\hat D(k))^2-(2d-1)}}\equiv \sum_{\sigma\in\{-1,1\}}\hat{g}_\sigma(k).
    \lbeq{bn-k-ident}
    \end{eqnarray}
Since, $\lambdahat_+(k)\neq \lambdahat_-(k)$ for $k$ small, \refeq{bn-k-ident} holds in particular for $k$ small.
A straightforward computation gives
    \begin{eqnarray*}
    \nabla^2g_\sigma(k)\Big|_{k=0}&=&
    d\sigma  \left(\nabla^2\lambdahat_\sigma(0)\right) \frac  {   n\lambdahat^{n-1}_\sigma(0)\hat D(0) -(n-1)\lambdahat^{n-2}_\sigma(0)}{\sqrt{(d\hat D(0))^2-(2d-1)}}\\
    &&+d\sigma \left( \nabla^2 \hat D(0)\right)\Big( \frac  { \lambdahat^{n}_\sigma(0) }{\sqrt{(d\hat D(0))^2-(2d-1)}}
    - d^2\hat D(0) \frac  {\lambdahat^{n}_\sigma(0)\hat D(0) -\lambdahat^{n-1}_\sigma(0)}{\left(\sqrt{(d\hat D(0))^2-(2d-1)}\right)^3}\Big)\\
    &=& -d\sigma \frac {d ((1+\sigma)d-1)}{(d-1)} \frac  {   n\lambdahat^{n-1}_\sigma(0) -(n-1)\lambdahat^{n-2}_\sigma(0)}{d-1}\\
    &&+\frac{d\sigma}{d-1}\lambdahat^{n-1}_\sigma(0)\left( \lambdahat_\sigma(0)
    - d^2 \frac{\lambdahat_\sigma(0)-1}{(d-1)^2}\right),
    \end{eqnarray*}
so that
    \begin{eqnarray*}
    \nabla^2g_1(k)\Big|_{k=0}
    &=& - \frac {d}{d-1}\frac{(2d-1)^{n-1}}{d-1}\left[2d(d-1)n+(4d-1)\right]\\
    \nabla^2g_{-1}(k)\Big|_{k=0}&=& \frac {d}{d-1}
    \left[\frac {d (-1)}{(d-1)} - 1\right]=-d\frac{2d-1}{(d-1)^2}.
    \end{eqnarray*}
We arrive at
    \begin{eqnarray*}
    -\frac 1 {\hat b_n(0)} \nabla^2 \hat b_n(0)&=&\frac {1}{2d(2d-1)^{n-1}}
    \sum_{\sigma\in\{-1,1\}}-\nabla^2g_\sigma(k)\Big|_{k=0}\\
    &=&\frac{d}{d-1}n +\frac{4d-1}{2(d-1)^2}+\frac{d}{2(d-1)^2(2d-1)^{n-2}}.
    \end{eqnarray*}
\end{proof}
\begin{lemma}[Tightness]
\label{tightnessLemma}
For $d\geq 2$ and $0\leq t_1<t_2<t_3\leq 1$, there exists a $K > 0$, such that, for all $n\geq 1$,
    \begin{eqnarray*}
    \expec[ \|X_n(t_2)-X_n(t_1)\|^{2}_2 \|X_n(t_3)-X_n(t_2)\|^{2}_2 ]\leq K (t_2-t_1)(t_3-t_2).
    \end{eqnarray*}
\end{lemma}

\begin{proof}{}
We use the same notation as in the proof of Lemma \ref{FCLTpointwise}.
In \refeq{splittingIntoSmallPieces}, we have seen how to describe the number
of NBWs that visit a number of fixed points. This time we forget about the
non-backtracking constraint between two subsequent NBWs to upper bound
    \begin{eqnarray*}
    &&\!\!\!\!\!\!\!\!\!\!\!\!\!\!\!\!\!\!\!\!\!\!\!\!\!\!\!\!\!\!\expec[ \|\omega_{\lfloor t_2 n \rfloor}-\omega_{\lfloor t_1 n \rfloor}\|^{2}_2 \|\omega_{\lfloor t_3 n \rfloor}-\omega_{\lfloor t_2 n \rfloor}\|^{2}_2 ]\\
    &\leq &\frac {1} {\hat b_n(0)} \sum_{x_1,x_2,x_3,x_4\in\Zd} \|x_2-x_1\|_2^{2}\|x_3-x_2\|_2^{2} b_{\eta_1(n)}(x_1) \prod_{r=2}^4 b_{\eta_r(n)}(x_r-x_{r-1}) \\
    &=& \frac {1} {\hat b_n(0)} \sum_{x_1,x_2,x_3,x_4\in\Zd} \|x_2\|_2^{2}\|x_3\|_2^{2} \prod_{r=1}^4 b_{\eta_r(n)}(x_r)\\
    &=& \frac {\hat b_{\eta_1(n)} \hat b_{\eta_4(n)}}{b_n} \sum_{x_2\in\Zbold} b_{\eta_2(n)}(x_2)\|x_2\|_2^2\sum_{x_3\in\Zbold} b_{\eta_3(n)}(x_3) \|x_3\|_2^2\\
    &=& \left(\frac {2d-1} {2d}\right)^3 \expec[\|\omega_{\eta_2(n)}\|_2^2]\expec[\|\omega_{\eta_3(n)}\|_2^2].
    \end{eqnarray*}
Applying Lemma \ref{limitedsecondmoment} completes the proof.
\end{proof}

\subsection{Extension to non-nearest-neighbor setting}
\label{sec-non-nn}
In this section, we extend the analysis of NBW on $\Z^d$ to other bond sets.
We start by introducing the bond sets that we consider. We let $\mathbb{B}\subset \Z^d\times \Z^d$ be a translation invariant collection of bonds. Let
$\mathbb{V}_0=\{x\colon \{0,x\}\in \mathbb{B}\}$ denote
the set of endpoints of bonds containing the origin and write $\degree=|\mathbb{V}_0|$. We assume that
$0\not\in \mathbb{V}_0$, and that
$\mathbb{V}_0$ is symmetric, i.e., $-x\in \mathbb{V}_0$
for every $x\in \mathbb{V}_0$. Thus, $\degree$ is even.
We define the simple random walk step
distribution by
    \eqn{
    \lbeq{D-gen}
    D(x)=\frac{1}{\degree} \indic{x\in \mathbb{V}_0}.
    }

Define the matrices $\mC,\mJ\in\Cbold^{\degree\times \degree}$
by $(\mC)_{x,y}=1$  and
$(\mJ)_{x,y}=\delta_{x,-y}$, and let
the diagonal matrix $\mD[k]$ have entries
$(\mD[k])_{x,x}=\e^{\ii k\cdot x},$
where $x, y\in {\mathbb{B}}_0$.
Then, we define the matrix $\mA[k]$ of size $\degree\times\degree$ by
    \eqn{
    \mA[k]=(\mC-\mJ)\mD[-k].
    }
With this definition at hand, we see that
\refeq{NBWRecSchemeFourierVec1}--\refeq{NBWRecSchemeFourierVec3}
remain to hold. As a result, also Lemmas \ref{EVrighthard}--\ref{EVrightsimple},
whose proof only depends on \refeq{NBWRecSchemeFourierVec1}--\refeq{NBWRecSchemeFourierVec3},
continue to hold when we replace each occurrence of $2d$ by $m$.
Since the proof of Lemma \ref{qncharac}, in turn, only
depends on Lemmas \ref{EVrighthard}--\ref{EVrightsimple}, also it extends to
this setting, so that, for example
    \begin{equation}
    \lambdahat_\pm (k)=F_{\pm}(\hat{D}(k);\degree),
    \qquad
    \text{where}
    \qquad
    F_{\pm}(x;\degree)=\frac {1}{2}\left(\degree x\pm \sqrt{(\degree x)^2-4(\degree-1)}\right),
    \end{equation}
and, when $\lambdahat_{1}(k)\neq \lambdahat_{-1}(k)$,
    \eqn{
    \hat b_n(k)=\frac{\degree}{2} \frac {\hat D(k)(\lambdahat^{n}_{1}(k)-\lambdahat^{n}_{-1}(k)) + (\lambdahat^{n-1}_{-1}(k)-\lambdahat^{n-1}_{1}(k))}{\lambdahat_{1}(k)-\lambdahat_{-1}(k)}.
    }
Naturally, Theorem \ref{FCLT} needs to be adapted, and now reads that
the processes $(X_n(t))_{t\geq 0}$
converge weakly to a Brownian motion with covariance matrix
${{\mathbf M}}$ of size $d\times d$, where, for $\iota, \kappa\in \{\pm1, \ldots, \pm d\}$,
we define
    \eqn{
    {\mathbf M}_{\iota,\kappa}
    =\frac{\partial^2\lambdahat_{1}(k)}{\partial k_{\iota}\partial k_{\kappa}}\Big|_{k=0} \lambdahat_1(0)^{-1}.
    }
We next compute ${{\mathbf M}}$ explicitly in terms of the covariance matrix of the
transition kernel $D$. We compute that $F_{+}(1;\degree)=\degree/2+(\degree-2)/2=\degree-1,$
and
    \begin{equation}
    \lbeq{F+-der-one}
    F_{+}'(x;\degree)=\degree /2+\frac{\degree^2 x}{2\sqrt{(\degree x)^2-4(\degree-1)}},
    \qquad
    \text{so that}
    \qquad
    F_{+}'(1;\degree)=\degree (\degree-1)/(\degree-2).
    \end{equation}
By symmetry, the odd derivatives of $\hat D(k)$ are zero, so that
a Taylor expansion yields
    \eqn{
     \hat D(k)
     =1-\tfrac{1}{2}k^T \mathbf{H} k+O(\|k\|_2^4),
     }
where, for $\iota, \kappa\in \{1, \ldots, d\}$,
    \eqn{
    \mathbf{H}_{\iota, \kappa}=\sum_{x} x_{\iota}x_{\kappa} D(x)
    }
denotes the covariance matrix of SRW. As a result,
    \eqn{
    {\mathbf M}=\mathbf{H} \frac{F_{+}'(1;\degree)}{F_{+}(1;\degree)}=\mathbf{H} \degree/(\degree-2).
    }
In the nearest-neighbor case, $\degree=2d$ and $\mathbf{H}=\mI/d$, so that
we retrieve the result in Theorem \ref{FCLT}.

\section{NBW on tori}
\label{secTorus}
In this section, we extend the results in Section \ref{secmodel} to
NBWs on tori. In Section \ref{sec-setting-torus} and \ref{sec-NBW-torus-asymp},
we investigate NBW on a torus of width $r\geq 2$, and in
Section \ref{secCube} we investigate NBW on the hypercube, for which $r=2$.
The study of random walks on various finite transitive graphs
has attracted considerable attention. See e.g.,
\cite{LevPerWil09} for a recent book on the subject,
and \cite{AldFil02} for a book in preparation.
Here we restrict ourselves to NBWs on tori.

\subsection{Setting}
\label{sec-setting-torus}
For $d\geq 2$ and $r\geq 3$, we
denote by $\Tbold=\Tbold_{r,d}=(\Z/r\Z)^d$ the discrete $d$-dimensional torus
with side length $r$. The torus has periodic boundaries, i.e., we identify
two points $x,y\in\Tbold_{r,d}$ if $x_i\text{ mod }r=y_i\text{ mod }r$ for
all $i=1,\dots,d$ where mod denote the modulus. We define the Fourier dual
torus of $\Tbold$ as
    \begin{eqnarray}
    \lbeq{Tstar-def}
    \Tbold^*_{r,d}:=\frac {2\pi} r \left\{ -\left\lfloor\frac {r-1} 2\right\rfloor, \dots ,\left\lceil\frac {r-1} 2\right\rceil  \right\}^d,
    \end{eqnarray}
so that each component of $k\in\Tbold^*_{r,d}$ is between $-\pi$ and $\pi$.
The Fourier transform of $f\colon \Tbold_{r,d} \to \Cbold$ is defined by
    \begin{eqnarray*}
    \hat f(k)=\sum_{x\in\Tbold_{r,n}}f(x)\e^{\ii k\cdot x}, \qquad \qquad k\in \Tbold^*_{r,d}.
    \end{eqnarray*}
As in Section \ref{secmodel}, we define an \emph{$n$-step random walk} on
$\Tbold_{r,d}$ to be an ordered tuple $\omega=(\omega_0,\dots,\omega_n)$, with $\omega_i\in\Tbold_{r,d}$ and $\omega_i-\omega_{i+1}\in \mathbb{V}_0$, where
we recall that $\mathbb{V}_0=\{x\colon \{0,x\}\in {\mathbb B}\}$,
and ${\mathbb B}$ is the translationally invariant bond set
on which our random walks moves. We always assume that $0\not\in \mathbb{V}_0$,
and that $\mathbb{V}_0$ is symmetric, i.e.,
if $x\in \mathbb{V}_0$, then also $-x\in \mathbb{V}_0$.
Further, we always assume that $\omega_0=(0,\dots,0)$.
The simple random walk step distribution is given by
    \begin{eqnarray}
    \lbeq{DefDTn}
    D(x)=\frac 1 {\degree} \indic{x\in \mathbb{V}_0}
    \qquad \text{ and }
    \qquad
    \hat D(k)=\frac {1} {\degree} \sum_{x\in \Tbold_{r,d}} \e^{\ii k\cdot x}.
    \end{eqnarray}
If an $n$-step random walk on $\Tbold_{r,d}$ additionally satisfies
$\omega_i\neq \omega_{i-2}$, then we call the walk a non-backtracking walk (NBW) on $\Tbold_{r,d}$.
Let $b_n(x)$ be the number of $n$-step NBWs with $\omega_n=x$. Further,
let $b^\iota_{n}(x)$ be the number of $n$-step NBWs $\omega$ with $\omega_n=x$
and $\omega_1\neq\ve[\iota]$.

In this setting, we can express $b_n(x)$ for NBW on $\Tbold_{r,d}$
in terms of NBW on $\Z^d$. Indeed, identify $\Tbold_{r,d}$ with $\{0, \ldots, r-1\}^d\subset \Z^d$,
and also identify $\mathbb{V}_0=\{x\colon \{0,x\}\in {\mathbb B}\}$ as a subset of $\{0, \ldots, r-1\}^d\subset \Z^d$. Define $\mathbb{V}_0^{\sss \Z}=\mathbb{V}_0 \cup (- \mathbb{V}_0)$
(which, by construction, are disjoint subsets of $\Zd$),
and define the random walk step distribution $D_{\sss \Z}(x)$ by the uniform distribution on
$\mathbb{V}_0^{\sss \Z}$. Then, for $x\in \Tbold_{r,d}$,
	\eqn{
	b_n(x)=\sum_{y\colon y\stackrel{r}{\sim}x} b_n^{\sss \Z}(y),
	}
where, for $y\in \Zd$ and $x\in \Tbold_{r,d}$, we say that $x\stackrel{r}{\sim}y$ when $x=y\mod r$,
and $b_n^{\sss Z}(y)$ denotes the number of $n$-step NBWs on $\Zd$ with step distribution $D_{\sss \Z}(x)$.
As a result, $\hat{b}_n(k)$ for NBW on $\Tbold_{r,d}$ is equal to
$\hat{b}_n^{\sss \Z}(k)$ for every $k\in \Tbold^*_{r,d}$ as defined in
\refeq{Tstar-def}. Therefore, we can use most results for
NBW on $\Z^d$ to study NBW on $\Tbold_{r,d}.$
We define the probability mass function of the endpoint of an $n$-step
NBW by
    \eqn{
    \label{pn-def}
    p_n(x)=\frac{b_n(x)}{\sum_{y} b_n(y)}
    =\frac{b_n(x)}{\degree(\degree-1)^{n-1}}.
    }


In order to study the asymptotic behavior of NBW, we
investigate $\lambdahat_{\pm}(k)$ for $k\neq 0$. Our main result in
this section is the following theorem:

\begin{theorem}[Pointwise bound on $\hat{b}_n(k)$]
\label{thm-bn-bd-torus}
Let $\mathbb{V}_0$ be symmetric and satisfy $0\not\in \mathbb{V}_0$. Then, for $n\in\Nbold$, NBW
with steps in $\mathbb{V}_0$ satisfies
    \eqn{
    |\hat{p}_n(k)|\leq \big(1/\sqrt{\degree-1}\vee |\hat{D}(k)|\big)^{n-1}.
    }
\end{theorem}

To prove Theorem \ref{thm-bn-bd-torus}, we start by
investigating $\lambdahat_{\pm}(k)$ for $k\neq 0$. For this,
we use Lemma \ref{EVrighthard} to note that
    \begin{equation}
    \lbeq{lambdapm-def}
    \lambdahat_\pm (k)=F_{\pm}(\hat{D}(k);\degree),
    \qquad
    \text{where}
    \qquad
    F_{\pm}(x;\degree)=\frac {1}{2}\left(\degree x\pm \sqrt{(\degree x)^2-4(\degree-1)}\right),
    \end{equation}
and where $\degree$ denotes the degree of our graph.
We bound $\lambdahat_\pm (k)$ in the following lemma:

\begin{lemma}[Bounds on $\lambdahat_\pm (k)$]
\label{lem-Fpm-bd}
For any $k\in \Tbold^*_{r,d}$,
    \eqn{
    |\lambdahat_+(k)|
    \begin{cases}
    =\sqrt{\degree -1} &\text{when }(\degree \hat{D}(k))^2-4(\degree-1)\leq 0;\\
    \leq (\degree-1)[1-(1-\hat{D}(k))\degree/(\degree-2)] &\text{when }\hat{D}(k)\geq 0,(\degree \hat{D}(k))^2-4(\degree-1)>0;\\
    \leq 1 &\text{when }\hat{D}(k)\leq 0, (\degree  \hat{D}(k))^2-4(\degree-1)>0.
    \end{cases}
    }
\end{lemma}

\begin{proof}
The function $x\mapsto F_{\pm}(x;\degree)$ is real when $(\degree x)^2-4(\degree-1)\geq 0$,
and complex when $(\degree x)^2-4(\degree-1)<0$. When $(\degree x)^2-4(\degree-1)<0$,
    \begin{equation}
    |F_{\pm}(x;\degree)|^2
    =\degree-1,
    \end{equation}
so that $|F_{\pm}(x;\degree)|=\sqrt{\degree-1}$.

When $(\degree x)^2-4(\degree -1)\geq 0$, by the symmetry
$F_{+}(x;\degree)=F_-(-x;\degree)$, we only need to investigate $x\in[0,1]$.
We start with $F_{-}(x;\degree)$, which clearly satisfies $F_{-}(x;\degree)\geq 0$.
Further, we can compute that $F_{-}(1;\degree)=1$, and
    \eqn{
    F_{-}'(x;\degree)=\degree /2-\frac{\degree^2 x}{2\sqrt{(\degree x)^2-4(\degree-1)}}
    =\frac{\degree}{2} \Big[1-\frac{\degree x}{\sqrt{(\degree x)^2-4(\degree-1)}}\Big]<0,
    }
so that $F_{-}(x;\degree)\leq 1$
for all $x\in [0,1]$ for which $(\degree x)^2-4(\degree-1)>0$.

To bound $F_+(x;\degree)$, we use \refeq{F+-der-one} as well as
    \eqan{
    F_{+}''(x;\degree)&=\frac{\degree^2}{2\sqrt{(\degree x)^2-4(\degree-1)}}
    -\frac{\degree^4x^2}{2((\degree x)^2-4(\degree-1))^{3/2}}\\
    &=\frac{\degree^2}{2((\degree x)^2-4(\degree-1))^{3/2}}
    \big\{((\degree x)^2-4(\degree-1))-(\degree x)^2\big\}\nn\\
    &=-\frac{2\degree^2(\degree-1)}{((\degree x)^2-4(\degree-1))^{3/2}}<0.\nn
    }
As a result, a Taylor expansion yields
    \eqan{
    F_{+}(x;\degree)
	&\leq F_{+}(1;\degree)+(x-1)F_{+}'(1;\degree)=(\degree-1)+(x-1)\degree (\degree-1)/(\degree-2)\\
    	&=(\degree-1)[1-(1-x)\degree/(\degree-2)].\nn
    }
\end{proof}

\noindent
{\it Proof of Theorem \ref{thm-bn-bd-torus}.}
By \refeq{NBWRecSchemeFourierVec1} and \refeq{NBWRecSchemeFourierVec3},
    \eqn{
    \lbeq{bn-op-bd}
    |\hat{b}_n(k)|\leq \|\mD[-k]\v1\|_2\|\vb_{n-1}(k)\|_2
    \leq \|\v1\|_2\|\mA[k]\|^{n-1}_{\rm \sss OP}\|\v1\|_2,
    }
where we write $\|{\mathbf M}\|_{\rm \sss OP}
=\sup \|{\mathbf M}x\|_2/\|x\|_2$ for the operator norm of
the matrix ${\mathbf M}$. We next use that $\mA[k]$ has
eigenvalues $\lambdahat_+(k),\lambdahat_-(k)$ and $\pm 1$
by Lemmas \ref{EVrighthard}-\ref{EVrightsimple}, so that
    \eqn{
    \|\mA[k]\|_{\rm \sss OP}=|\lambdahat_+(k)|\vee |\lambdahat_-(k)| \vee 1,
    }
where we use that for finite-dimensional matrices,
the operator norm is equal to the maximal eigenvalue, and for $x,y\in \R$, we write
$(x\vee y)=\max\{x,y\}$. Thus, we arrive at
	\eqn{
	|\hat{b}_n(k)|\leq \degree \big(|\lambdahat_+(k)|\vee |\lambdahat_-(k)| \vee 1\big)^{n-1}.
	}
By Lemma \ref{lem-Fpm-bd}, and since $\degree\geq 2$ so that $\sqrt{\degree -1}\geq 1$,
    \eqn{
    \frac{|\lambdahat_\pm(k)|}{\degree-1}
    \leq (\degree-1)^{-1/2}\vee  \big(1-[1-\hat{D}(k)]\frac{\degree}{\degree-2}\big)
    \leq (\degree-1)^{-1/2}\vee |\hat{D}(k)|.
    }
Substitution into \refeq{bn-op-bd} yields the claim.
\qed

\subsection{Asymptotics for NBW on the torus}
\label{sec-NBW-torus-asymp}
In this section, we study the convergence towards equilibrium of NBW
on tori of width $r\geq 3$. We focus on two different examples. The first is random walk on
products of complete graphs, where
    \eqn{
    \lbeq{B0-Hamming}
    \mathbb{V}_0=\{x\colon \exists ! i\in \{1, \ldots, d\} \text{ such that }x_i\neq 0\}.
    }
Our second example is NBW on the nearest-neighbor torus. The reason why
we study these cases separately is that NBW on products of complete graphs is
aperiodic, while nearest-neighbor NBW is periodic.
Therefore, the stationary distribution for
NBW on products of complete graphs equals the uniform distribution on
the torus, while for nearest-neighbor NBW, the parity of
the position after $n$ steps always equals that of $n$.
In Section \ref{secCube}, we further study random walk on the
$\degree$-dimensional hypercube.

For any small $\xi>0,$ we write $\Tm(\xi)$ for the $\xi$-{\em uniform mixing time} of NBW, that is,
    \eqn{
    \Tm(\xi) = \min \Big \{ n \colon  \max_{x,y} \,\,\frac {p_n(x,y) + p_{n+1}(x,y)}2
    \leq (1+\xi) V^{-1} \Big \},
    }
where $V=r^d$ is the volume of the torus.
We start to investigate NBW on products of complete graphs:

\paragraph{NBW on products of complete graphs.} Our main result is as follows:

\begin{lemma}
\label{lem-stat-distr-Hamming}
For every $d\geq 1$, $r\geq 3$, $n> \frac{d(r-1)}{r}\log{\big((r-1)/[(1+\xi)^{1/d}-1]\big)}$,
NBW on products of complete graphs satisfies that
    \eqn{
    \max_{x\in \Tbold_{r,d}}\big|p_n(x)-r^{-d}\big|
    \leq (\degree-1)^{-(n-1)/2}+\xi r^{-d}.
    }
In particular, for every $\vep>0$, there exists $V_0$ such that when
$r^d\geq V_0$,
	\[\Tm(\xi)\leq  (1+\vep)\frac{d(r-1)}{r}\log{\big((r-1)/[(1+\xi)^{1/d}-1]\big)}.\]
\end{lemma}

When $\xi$ is quite small, we obtain that $\Tm(\xi)\leq (1+2\vep)\frac{r}{r-1}\log{(d(r-1)/\xi)}$.

\begin{proof} The inverse Fourier transform on $\Tbold_{r,d}$ is given by
    \begin{eqnarray*}
    f(x)=\frac{1}{r^d}\sum_{k\in \Tbold^*_{r,d}}\hat{f}(k) \e^{\ii k\cdot x},
    \end{eqnarray*}
so that
    \eqn{
    b_n(x)=\frac{1}{r^d}\sum_{k\in \Tbold^*_{r,d}}\hat{b}_n(k) \e^{\ii k\cdot x}
    =r^{-d}\hat{b}_n(0)+\frac{1}{r^d}\sum_{k\in \Tbold^*_{r,d}\colon k\neq \vec{0}}
    \hat{b}_n(k) \e^{\ii k\cdot x}.
    }
Therefore,
    \eqn{
    \big|p_n(x)-r^{-d}\big|
    \leq
    \frac{1}{r^d}\sum_{k\in \Tbold^*_{r,d}\colon k\neq \vec{0}}
    \frac{|\hat{b}_n(k)|}{\hat{b}_n(0)}.
    }
To bound $|\hat{b}_n(k)|$, we rely on Theorem \ref{thm-bn-bd-torus}, and
start by computing
    \eqn{
    \hat{D}(k)=\frac{1}{d(r-1)}\sum_{i=1}^d\sum_{j=1}^{r-1} \e^{\ii k_i j}.
    }
Since $k_i\in \frac {2\pi} r \left\{ -\left\lfloor\frac {r-1} 2\right\rfloor, \dots ,\left\lceil\frac {r-1} 2\right\rceil  \right\}$, we have that $\sum_{j=0}^{r-1} \e^{\ii k_i j}=0$ for $k_i\neq 0$.
Therefore,
    \eqn{
    \hat{D}(k)=\frac{1}{d(r-1)}\sum_{i=1}^d\sum_{j=1}^{r-1} \e^{\ii k_i j}
    =\frac{1}{d(r-1)}\sum_{i=1}^d (r\indic{k_i=0}-1)
    =1-\frac{r}{d(r-1)}a(k),
    }
where $a(k)=\sum_{i=1}^d \indic{k_i\neq 0}$
denotes the number of non-zero coordinates of $k$. By this observation,
$\hat D(k(j))=1-\frac{rj}{d(r-1)}$ for
any $k(j)\in \Tbold^*_{r,d}$ for which $a(k(j))=j$.
Then, by Theorem \ref{thm-bn-bd-torus} and the fact that there are
${{d}\choose{j}}(r-1)^j$ values of $k\in \Tbold^*_{r,d}$ for which $a(k(j))=j$,
    \begin{eqnarray}
    \lbeq{formForbn-Hamming}
    \big|p_n(x)-r^{-d}\big|
    &=&r^{-d}\sum_{j=1}^{d} (1/\sqrt{\degree -1}\vee |\hat D(k(j))|)^{n-1}
    {{d}\choose{j}}(r-1)^j\\
    &\leq& r^{-d} \sum_{j=1}^{d}  {{d}\choose{j}}(r-1)^j \Big[\big|1-\frac{rj}{d(r-1)}\big|^{n-1}+(\degree-1)^{-(n-1)/2}\Big]\nn\\
    &\leq& (\degree-1)^{-(n-1)/2}+r^{-d} \sum_{j=1}^{d}  {{d}\choose{j}}(r-1)^j\e^{-rj(n-1)/[d(r-1)]}\nn,
    \end{eqnarray}
where we use that $|1-\frac{rj}{d(r-1)}|\leq \e^{-rj(n-1)/[d(r-1)]}$ for any $j=1, \ldots, d$.
Thus, 
    \eqan{
    \big|p_n(x)-r^{-d}\big|
    &\leq (\degree-1)^{-(n-1)/2}+r^{-d} \big[(1+(r-1)\e^{-r(n-1)/[d(r-1)]})^d-1\big]\nn\\
    &\leq (\degree-1)^{-(n-1)/2}+\xi r^{-d},
    }
when $n> \frac{d(r-1)}{r}\log{\big((r-1)/[(1+\xi)^{1/d}-1]\big)}$. The result on $\Tm(\xi)$
follows immediately.
\end{proof}

\paragraph{NBW on the nearest-neighbor torus.} Our main result is as follows:

\begin{lemma}
\label{lem-stat-distr-nn-torus}
For every $d\geq 1$, $r\geq 3$ and $n>\log{(2/[(1+\xi/2)^{1/d}-1]})/[1-\cos(2\pi/r)]$,
NBW on the $d$-dimensional nearest-neighbor torus satisfies
that
    \eqn{
    \max_{x\in \Tbold_{r,d}}\big|p_n(x)-[1+(-1)^{\|x\|_1+n}]r^{-d}\big|
    \leq (2d-1)^{-n/2}+\xi r^{-d}.
    }
In particular, for every $\vep>0$, there exists $V_0$
such that whenever $r^d\geq V_0$,
	\[\Tm(\xi)\leq (1+\vep)\log{(2/[(1+\xi/2)^{1/d}-1]})/[1-\cos(2\pi/r)].\]
\end{lemma}

When $r^d$ is large and $\xi$ small, the above implies that
$\Tm(\xi)\leq (1+2\vep)(r^2/(2\pi^2)) (\log{(2d/\xi)})$.

\proof We adapt the proof of Lemma \ref{lem-stat-distr-Hamming} to this setting.
We have
    \eqn{
    b_n(x)=\frac{1}{r^d}\sum_{k\in \Tbold^*_{r,d}}\hat{b}_n(k) \e^{\ii k\cdot x}
    =r^{-d}[\hat{b}_n(0)+(-1)^{\|x\|_1}\hat{b}_n(\vec{\pi})]
    +\frac{1}{r^d}\sum_{k\in \Tbold^*_{r,d}\colon k\neq \vec{0},\vec{\pi}}
    \hat{b}_n(k) \e^{\ii k\cdot x}.
    }
Further, $\hat{D}(-\vec{\pi})=-1$, so that, by \refeq{qncharac2} $\hat{b}_n(\vec{\pi})=\hat{b}_n(0)(-1)^n.$
Therefore,
    \eqan{
    \big|p_n(x)-[1+(-1)^{\|x\|_1+n}]r^{-d}\big|
    &=\Big|\frac{1}{r^d}\sum_{k\in \Tbold^*_{r,d}\colon k\neq \vec{0},\vec{\pi}}
    \frac{\hat{b}_n(k)}{\hat{b}_n(0)} \e^{\ii k\cdot x}\Big|
    \leq \frac{1}{r^d}\sum_{k\in \Tbold^*_{r,d}\colon k\neq \vec{0},\vec{\pi}}
    \frac{|\hat{b}_n(k)|}{\hat{b}_n(0)}\nn\\
    &\leq (\degree-1)^{(n-1)/2} + \frac{1}{r^d}\sum_{k\in \Tbold^*_{r,d}\colon k\neq \vec{0},\vec{\pi}}
    |\hat{D}(k)|^{n-1}.\nn
    }
In the nearest-neighbor case, $\hat{D}(k-\vec{\pi})=-\hat{D}(k)$,
so we may restrict to $k$ for which $\hat{D}(k)\geq 0$. Let
    \eqn{
    \Tbold^*_{r,d,+}=\{k\in \Tbold^*_{r,d}\colon \hat{D}(k)\geq 0\}
    }
denote the set of $k$'s for which $\hat{D}(k)\geq 0$. Then,
    \eqan{
    \big|p_n(x)-[1+(-1)^{\|x\|_1+n}]r^{-d}\big|
    &\leq (\degree-1)^{(n-1)/2} + \frac{2}{r^d}\sum_{k\in \Tbold^*_{r,d,+}\colon k\neq \vec{0}}
    \hat{D}(k)^{n-1}.
    }
We use that
    \eqn{
    \hat{D}(k)=1-[1-\hat{D}(k)]\leq \e^{-[1-\hat{D}(k)]},
    }
so that
    \eqan{
    \big|p_n(x)-[1+(-1)^{\|x\|_1+n}]r^{-d}\big|
    &\leq (\degree-1)^{(n-1)/2} + \frac{2}{r^d}\sum_{k\in \Tbold^*_{r,d}\colon k\neq \vec{0}}
    \e^{-(n-1)[1-\hat{D}(k)]}\\
    &=(\degree-1)^{(n-1)/2} + \frac{2}{r^d}\Big[\sum_{k\in \Tbold^*_{r,d}}
    \e^{-(n-1)[1-\hat{D}(k)]}-1\Big]\nn\\
    &=(\degree-1)^{(n-1)/2} + \frac{2}{r^d}\Big[\Big(\sum_{k\in \Tbold^*_{r,1}}
    \e^{-(n-1)[1-\cos(k)]/d}\Big)^d-1\Big].\nn
    }
We use that the dominant contributions to the sum over $k\in \Tbold^*_{r,1}$
comes from $k=0$ and $k=\pm 2\pi/r$, so that
    \eqn{
    \sum_{k\in \Tbold^*_{r,1}}
    \e^{-(n-1)[1-\cos(k)]/d}=1+2\e^{-(n-1)[1-\cos(2\pi/r)]/d)}(1+o(1)),
    }
so that
    \eqan{
    \big|p_n(x)-[1+(-1)^{\|x\|_1+n}]r^{-d}\big|
    &\leq (\degree-1)^{(n-1)/2} +\frac{2}{r^d}\Big[\Big(1+2\e^{-(n-1)[1-\cos(2\pi/r)]/d)}(1+o(1))\Big)^d-1\Big]\nn\\
    &\leq  (2d-1)^{-(n-1)/2}+\xi r^{-d}
    }
when $n>(\log{(2/[(1+\xi/2)^{1/d}-1]})/[1-\cos(2\pi/r)]$.
\qed

\subsection{NBW on the hypercube}
\label{secCube}
In this section we specialize the results of Sections \ref{secmodel} and Sections
\ref{sec-setting-torus}--\ref{sec-NBW-torus-asymp} to
the hypercube $\Tbold_{2,\degree}= \{0,1\}^\degree$.
The results in this section are an important ingredient to the
analysis of percolation on the hypercube in \cite{HofNac11}.
We start with some notation. It will be convenient to let the Fourier dual space of
$\{0,1\}^\degree= \{0,1\}^\degree$ be $\Qbold^*_\degree= \{0,1\}^\degree$, so that the
Fourier transform of a summable function $f\colon \{0,1\}^\degree\to \Cbold$ is given by
    \begin{eqnarray*}
    \hat f(k)=\sum_{x\in\{0,1\}^\degree}f(x)\e^{\ii \pi k\cdot x}
    =\sum_{x\in\{0,1\}^\degree}f(x)(-1)^{k\cdot x}, \qquad (k\in \{0,1\}^\degree).
    \end{eqnarray*}
For $k\in\{0,1\}^\degree$, let $a(k)$ be its number of non-zero entries. Then,
the SRW step distribution on $\{0,1\}^\degree$ satisfies
    \begin{eqnarray}
    \lbeq{DefhatDK}
    \hat D(k)=\frac 1 \degree \sum_{j=1}^\degree (-1)^{k_j}= 1-2a(k)/\degree.
    \end{eqnarray}

The main result of this section is as follows:

\begin{theorem}[NBW on hypercube]
\label{thm-NBW-hyper}
For NBW on the hypercube $\{0,1\}^\degree$,
    \eqn{
    \label{pn-bd-hyper}
    \hat{p}_n(k)\leq \Big(|\hat{D}(k)|\vee 1/\sqrt{\degree -1}\Big)^{n-1}.
    }
Consequently, for every $\vep>0$, there exists $\degree_0$ such that for all $\degree\geq \degree_0$,
    \eqn{
    \Tm(\xi)\leq \frac{\degree(1+\vep)}{2}\log{(2\degree/\xi)}.
    }
\end{theorem}

Random walks on hypercubes have attracted considerable attention,
In fact, the bound on the uniform mixing time for NBW
in discrete time closely matches the one for SRW
in continuous time (see \cite[Lemma 2.5(a)]{Aldo83b}).

\proof
The bound in \eqref{pn-bd-hyper} follows directly from Theorem \ref{thm-bn-bd-torus} .
We continue to investigate the convergence of the
NBW transition probabilities $x\mapsto p_n(x)$
to its quasi-stationary distribution, which is $2\times 2^{-\degree}$ when
$n$ and $x$ have the same parity and $0$ otherwise.
Here we say that $n$ and $x$ have the same parity when there
exists an $n$-step path from $0$ to $x$.

\begin{lemma}[Convergence to equilibrium for NBW on hypercube]
\label{lem-stat-distr-hyper}
For NBW on $\{0,1\}^{\degree}$ and every $n>d(\log{d}+\log{\xi})/2$,
    \eqn{
    \max_{x\in \{0,1\}^{\degree}}\big|p_n(x)-2^{-\degree}[1+(-1)^{\|x\|_1+n}]\big|
    \leq (\degree-1)^{-(n-1)/2}+\xi 2^{-\degree}.
    }
Consequently, for every $\vep>0$, there exists $\degree_0$ such that
for every $\degree\geq \degree_0$,
	\[\Tm(\xi)\leq -\frac{\degree(1+\vep)}{2}\log{([1+\xi/2)]^{1/\degree}-1)}.\]
\end{lemma}

\begin{proof} 
The inverse Fourier transform on the $\degree$-dimensional hypercube is given by
    \begin{eqnarray*}
    f(x)=2^{-\degree}\sum_{k\in \{0,1\}^\degree}\hat{f}(k) (-1)^{k\cdot x},
    \end{eqnarray*}
so that, using $\hat{b}_n(\v1)=(-1)^n \hat{b}_n(0)$,
    \eqn{
    b_n(x)=2^{-\degree}\sum_{k\in \{0,1\}^\degree} \hat{b}_n(k)(-1)^{k\cdot x}
    =2^{-\degree}[1+(-1)^{\|x\|_1+n}]\hat{b}_n(0)
     +\sum_{k\in \{0,1\}^\degree\colon k\neq \v1, \vec{0}} \hat{b}_n(k)(-1)^{k\cdot x}.
    \lbeq{formForbn}
    }
Substituting the bound \eqref{pn-bd-hyper} in \refeq{formForbn} leads to
    \eqan{
    \big|p_n(x)-2^{-\degree}[1+(-1)^{\|x\|_1+n}]\big|
    &\leq 2\cdot 2^{-\degree} \sum_{j=1}^{\degree/2}  {{\degree}\choose{j}}
    \big[(1-2j/\degree)^{n-1}+(\degree-1)^{-n/2}\big]\\
    &\leq (\degree-1)^{-n/2}+2\cdot 2^{-\degree} \sum_{j=1}^{\degree/2}  {{\degree}\choose{j}}\e^{-2j(n-1)/\degree}\nn\\
    &\leq (\degree-1)^{-n/2}+2\cdot 2^{-\degree} \big[(1+\e^{-2(n-1)/\degree})^\degree-1\big]\nn\\
    &\leq  (\degree-1)^{-n/2}+\xi 2^{-\degree},\nonumber
    }
when $n>-\frac{\degree}{2}\log{([1+\xi/2]^{1/\degree}-1)}$.
\end{proof}

\paragraph{Acknowledgements.}
This work was supported in part by the Netherlands
Organization for Scientific Research (NWO). RvdH thanks
Asaf Nachmias for valuable comments on an early version of
the paper, and Gordon Slade for pointing us at the relevant
literature.

{\small \bibliographystyle{plain}

\begin{thebibliography}{10}

\bibitem{Aldo83b}
D.~Aldous.
\newblock Minimization algorithms and random walk on the {$d$}-cube.
\newblock {\em Ann. Probab.}, {\bf 11}(2):403--413, (1983).

\bibitem{AldFil02}
D.~Aldous and J.~Fill.
\newblock {\em Reversible Markov Chains and Random Walks on Graphs} (monograph in
  preparation.), (2002).

\bibitem{AloBenLubSod07}
N.~Alon, I.~Benjamini, E.~Lubetzky, and S.~Sodin.
\newblock Non-backtracking random walks mix faster.
\newblock {\em Commun. Contemp. Math.}, {\bf 9}(4):585--603, (2007).

\bibitem{AloLub09}
N.~Alon and E.~Lubetzky.
\newblock Poisson approximation for non-backtracking random walks.
\newblock {\em Israel J. Math.}, {\bf 174}:227--252, (2009).

\bibitem{Bill95}
P.~Billingsley.
\newblock {\em Probability and Measure}.
\newblock Wiley Series in Probability and Mathematical Statistics. John Wiley
  \& Sons Inc., New York, third edition, (1995).
\newblock A Wiley-Interscience Publication.

\bibitem{HofNac11}
R.~van~der Hofstad and A.~Nachmias.
\newblock Hypercube percolation.
\newblock Preprint (2012).

\bibitem{KotSun00}
M.~Kotani and T.~Sunada.
\newblock Zeta functions of finite graphs.
\newblock {\em J. Math. Sci. Univ. Tokyo}, {\bf 7}(1):7--25, (2000).

\bibitem{LevPerWil09}
D.~Levin, Y.~Peres, and E.~Wilmer.
\newblock {\em Markov Chains and Mixing Times}.
\newblock American Mathematical Society, Providence, RI, (2009).
\newblock With a chapter by James G. Propp and David B. Wilson.

\bibitem{MadSla93}
N.~Madras and G.~Slade.
\newblock {\em The Self-Avoiding Walk}.
\newblock Birkh{\"a}user, Boston, (1993).

\bibitem{Noon98}
J.~Noonan.
\newblock New upper bounds for the connective constants of self-avoiding walks.
\newblock {\em J. Statist. Phys.}, {\bf 91}(5-6):871--888, (1998).

\bibitem{OrtWoe07}
R.~Ortner and W.~Woess.
\newblock Non-backtracking random walks and cogrowth of graphs.
\newblock {\em Canad. J. Math.}, {\bf 59}(4):828--844, (2007).

\bibitem{PonTit00}
A.~P{\"o}nitz and P.~Tittmann.
\newblock Improved upper bounds for self-avoiding walks in {${\bf Z}^d$}.
\newblock {\em Electron. J. Combin.}, {\bf 7}:Research Paper 21, 10 pp.
  (electronic), (2000).

\bibitem{Slad06}
G.~Slade.
\newblock {\em The lace expansion and its applications}, volume~{\bf 1879} of
  {\em Lecture Notes in Mathematics}.
\newblock Springer-Verlag, Berlin, (2006).
\newblock Lectures from the 34th Summer School on Probability Theory held in
  Saint-Flour, July 6--24, 2004, Edited and with a foreword by Jean Picard.

\end{thebibliography}
\def\cprime{$'$}

}
\end{document}